\documentclass[12pt]{article}
\usepackage{amsmath,amsthm,amsfonts,amssymb}
\usepackage{fullpage}
\newtheorem{theorem}{Theorem}[section]
\newtheorem{lemma}{Lemma}[section]
\newtheorem{cor}{Corollary}[section]
\newtheorem{pro}{Proposition}[section]
\newtheorem{defn}{Definition}[section]

\DeclareMathOperator{\inte}{int}

\title{Bounds for a solution set of linear complementarity problems over Hilbert spaces }

\author{Projesh Nath Choudhury \thanks{ Department of Mathematics, Indian Institute of Sciences, Bengaluru, India. Email:  projeshc@iisc.ac.in, projeshnc@alumni.iitm.ac.in}  \and M. Rajesh Kannan\thanks{Department of Mathematics, Indian Institute of Technology Kharagpur, Kharagpur, India. Email:  rajeshkannan1.m@gmail.com,rajeshkannan@maths.iitkgp.ac.in } \and  K.C. Sivakumar\thanks{Department of Mathematics, Indian Institute of Technology Madras, Chennai, India. Email: kcskumar@iitm.ac.in}
}

\date{\today}
\begin{document}
\maketitle

\begin{abstract}
Let $H$ be a real Hilbert space. In this short note, using some of the properties of bounded linear operators with closed range defined on $H$, certain bounds for a specific convex subset of the solution set of infinite linear complementarity problems, are established. 
\end{abstract}



\newpage
\section{Introduction}\label{notation}
Let $H$ be a Hilbert space over the real field $\mathbb{R}$. Let $\mathcal{B}(H)$ denote the set of bounded linear operators on the Hilbert space $H$. A subset $K$ of a Hilbert space $H$ is said to be a cone, if $x, y \in K$ and $\lambda \geq 0$ imply that $x+ y \in K$, $\lambda x \in K$ and $K \cap (-K) = \{0\}.$ Let $H^{'}$ denote the space of all continuous linear functionals on $H$. The dual cone $K^*$ of a cone $K$ is defined as follows: $K^{*}=\{f \in H^{'} : f(x) \geq 0 , \forall x \in K \}$. Let $H_+$ be a cone in $H$. For a given vector $q \in H$ and an operator $T \in \cal{B}$ $(H)$ the linear complementarity problem, written as $LCP(T,q)$, is to find a vector $z \in H_+$ such that $Tz + q \in H_+^*$ and $\langle z, Tz+q \rangle = 0 $. For the case $H=\mathbb{R}^n$ and $H_+ =\mathbb{R}_+^n$ (entry-wise nonnegative vectors in $\mathbb{R}^n$), this problem is well studied \cite{cot-pan-sto}. The linear programming problem in $\mathbb{R}^n$ is defined as follows:  For a given vector $p \in \mathbb{R}^n$, find a vector  $x \in \mathbb{R}_+^n$ which minimizes $p^Tx$ subject to $Tx + q \in \mathbb{R}_+^n$, where $T$ is an $n \times n$ real matrix, $q$ is a vector in $\mathbb{R}^n$ and $p^T$ denotes the transpose of the vector $p$. In \cite{manga1}, the author proved that, under certain assumptions, each solution of linear programming problem is also a solution of the linear complementarity problem. In \cite{cry-demp}, the authors extended the results of \cite{manga1} to infinite dimensional Hilbert spaces. Also linear complementarity problems are closely related to the variational inequality problems. In \cite{cry-demp}, the authors established  that equivalence linear complementarity problem and variational inequality problem. In \cite{isac-numerical}, the author studied the boundedness of solution set of  linear complementarity problems over infinite dimensional Hilbert spaces. One of the main objectives of this article is to study some of the properties of closed range operators. Using these properties, we shall extend some of the bounds of the solution set of linear complementarity problems established in   \cite{isac-numerical}. Also, we provide an alternate simple proof for one of the main results (Theorem $3.1$) of  \cite{isac-numerical}.

This article is organized as follows: In section \ref{ndp}, we collect some of the known results. In section \ref{prop-clos}, we establish some of the properties of closed range operators. In Theorem \ref{posi-closed} and Theorem \ref{posi-closed-conv}, we derive an equivalent condition for an operator to be a closed range positive semidefinite operator.  In section \ref{bound_lcp}, we derive bounds for a convex subset of the solution set of linear complementarity problems.

\section{Notation, Definitions and Preliminary Results}\label{ndp}
For an operator $T \in \mathcal{B}(H)$,  let $T^{*}, R(T)$ and $N(T)$ denote the adjoint, range space and null space of $T$, respectively.  The \emph{Moore-Penrose inverse} of an operator $T  \in \mathcal{B}(H)$ is the unique operator, if it exists, $S \in \mathcal{B}(H)$ satisfying the following conditions: $(1)~ T = TST$,  $(2)~S = STS$, $(3)~(TS)^{*} = TS$ and $(4)~(ST)^{*} = ST$, and is denoted by $T^{\dagger}$.  For a subset $M$ of  $H$, $\overline{M}$ denotes the topological closure of $M$.  An operator $T \in  \mathcal{B}(H)$ is said be \emph{closed range operator}, if $R(T)$ is closed. It is well known that, an operator  $T \in \mathcal{B}(H)$ has Moore-Penrose inverse $S \in \mathcal{B}(H)$ if and only if $R(T)$ is closed \cite{benis, kato}. A self-adjoint operator $T  \in \mathcal{B}(H)$ is said to be \emph{positive semidefinite} if for all $x \in H$, $\langle Tx, x \rangle \geq 0$, and $T$ is said to be \emph{positive definite} if $\langle Tx , x \rangle > 0$ for all nonzero $x \in H$ \cite{isac-numerical}. The positive semidefinite operators are known as positive operators in the literature \cite{reed-simon}. If $T$ is a positive semidefinite operator, then there exist a unique  positive semidefinite operator $S$ such that $S^2 = T$ \cite[Lemma 3.2]{bott}. The operator $S$  is called the positive square root of the operator $T$ and is denoted by $T^\frac{1}{2}.$

If $T \in \mathcal{B}(H)$ and $R(T)$ is closed, then the following holds: $(a)~ R(T^*) = R(T^\dagger)$, $(b)~TT^\dagger y = y$ for all $y \in R(T)$ and $(c)~TT^\dagger = T^\dagger T$, whenever $R(T) = R(T^*)$. Further, if $T$ is positive semidefinite, then  $R(T) = R(T^\frac{1}{2}) , N(T) = N(T^\frac{1}{2})$.

Let $X$ be a real linear space. Then $X$ is called a {\it partially ordered vector space} if there is a partial order $"\leq"$ defined on $X$ such that the following compatibility conditions are satisfied: $(i)~ x \leq y \Longrightarrow x + z \leq y + z$ for all $z \in X$ and $(ii)~ x \leq y \Longrightarrow \alpha x \leq \alpha y$ for all $\alpha \geq 0$. A subset $X_+$ of a real linear space $X$ is said to be a {\it cone} if, $X_+ + X_+ \subseteq X_+,$  $\alpha X_+ \subseteq X_+ $ for all $ \alpha \geq 0,$  $X_+ \cap {-X_+} = \{0\}$ and $X_+ \neq \{0\}$. A vector $x \in X$ is said to be {\it nonnegative}, if $x \in X_+$. This is denoted by $x \geq 0$. We define  $x \leq y$ if and only if $y-x \in X_+$. Then $"\leq"$ is a partial order (induced by $X_+$) on $X$. Conversely, if $X$ is a partially ordered normed linear space with the partial order $"\leq"$, then the set $X_
+ =\{x \in X : x \geq 0\}$ is a cone, and it is called the {\it positive cone} of $X$. By a {\it partial ordered  real normed linear space} $X$ we mean a real normed linear space $X$ together with a {\it closed} positive cone $X_+$. Let $X'$ denote the space of all continuous linear functionals on $X$. The dual cone $X_+^*$ of $X_+$, is defined as follows: $X_+^{*}=\{f \in X' : f(x) \geq 0 , \forall x \in X_+ \}$.

A partially ordered real normed linear space which is also a Banach space is called a {\it partially ordered Banach space}. A partially ordered real normed linear space which is also a Hilbert space is called a {\it partially ordered Hilbert space}.
A cone $X_+$ on a real normed linear space $X$ is said to be solid if $\inte(X_+) \neq \emptyset $, where $\inte(X_+)$ denotes the set of all interior points of $X_+$. If $X$ is a Hilbert space, then a cone $X_+$ is said to be self-dual, if $X_+ = X_+^*$.

\section{Properties of closed range operators}\label{prop-clos}
In this section we study some of the properties of closed range operators.
\begin{defn}
For an operator $T \in \mathcal{B}(H)$ , define $M(T) = \sup \{\langle Tx, x \rangle : x \in H , ||x|| = 1\},$ $m(T) = \inf \{\langle Tx, x \rangle : x \in H , ||x|| = 1\}$ and $m_r(T) = \inf \{\langle Tx, x \rangle : x \in R(T^*) , ||x|| = 1\}.$
\end{defn}

\begin{pro}
Let $T  \in \mathcal{B}(H)$. Then the following statements hold:
\begin{enumerate}
\item[(i)] An operator $T$ is self adjoint if and only if $T^\dagger$ is self adjoint,
\item[(ii)]  For an operator $T$,  $m(T) \geq 0 $ if and only if $T$ is positive semidefinite,
\item[(iii)] A closed range operator $T$ is  self-adjoint positive semidefinite  if and only if $T^\dagger$ is self-adjoint  positive semidefinite.
\item[(iv)] For a closed range operator $T$, $m_r(T) > 0$ if and only if $m_r(T^{\dagger})>0$.
\end{enumerate}
\end{pro}
\begin{proof}
Proofs of part $(i)$ and part $(ii)$ are easy to verify.

\textbf{$(iii)$} Let $T$  be a self adjoint positive semidefinite operator with  closed range. Then, $T^{\dagger}$ exists, and $T^\dagger$ is self adjoint. Let $x \in H$.  Then, $\langle T^\dagger x, x \rangle =  \langle T^\dagger T T^\dagger x , x \rangle=\langle T T^\dagger x , T^\dagger x \rangle \geq0$, since $T$ is positive semidefinite. Thus $T^\dagger$ is positive semidefinite. The converse can be proved in a similar way.

\textbf{$(iv)$} Let $x \in R(T^{*})$, then $x = T^{\dagger}y$ for some $y \in R(T)$. Now,  $\langle  Tx , x \rangle = \langle TT^\dagger y, T^\dagger y\rangle = \langle y, T^\dagger y\rangle$. Since,  $H$  is a Hilbert space over the field of real numbers, we have $ \langle y, T^\dagger y\rangle=  \langle T^\dagger y, y\rangle$. Thus, $\langle  Tx , x \rangle  =  \langle T^\dagger y, y\rangle$, for some $y \in R(T)$, and hence we have $m_r(T)>0$ if and only if $m_r(T^\dagger) >0$.

\end{proof}

The following lemma will be useful in the proof of the subsequent result.

\begin{lemma}\label{mrT-closure}
Let $T \in \mathcal{B}(H)$ be a self-adjoint operator. Then $m_r(T) = \inf \{\langle Tx, x \rangle : x \in \overline{R(T)} , ||x|| = 1\}.$
\end{lemma}
\begin{proof}
 Let $x \in \overline{R(T)}$ and $||x||=1$. Then, there exists a sequence of nonzero vectors $\{x_n\}$ in $R(T)$ such that $x_n$ converges to $x$. Now, since $||x|| =1 $, $x_n$ converges to $x$ implies that $\frac{x_n}{||x_n||}$ converges to $x$. So, without loss of generality, we can assume that $||x_n|| = 1$  for all $n$. Now, $|\langle Tx_n, x_n \rangle  - \langle Tx, x \rangle|=  |\langle  Tx_n, (x_n-x)\rangle + \langle T(x_n-x), x \rangle |\leq ||T|| ||x_n||||x_n-x|| + ||T||||x_n-x||||x||.$ Thus, $x_n$ converges to $x$ implies $\langle Tx_n, x_n \rangle$  converges to  $\langle Tx, x \rangle$, and hence $\langle Tx, x \rangle \geq m_r(T)$.
 \end{proof}

In the next theorem, we establish a sufficient condition for the positive semidefiniteness of the operator $T$ in terms of $m_r(T)$.

\begin{theorem}\label{posi-closed}
Let $T \in \mathcal{B}(H)$ be a self-adjoint operator. If  $m_r(T) > 0$, then $T$ is positive semidefinite and $R(T)$ is closed.
\end{theorem}
\begin{proof}
 By Lemma \ref{mrT-closure}, $\langle Tx, x \rangle \geq 0$ for all $x \in \overline{R(T)}$. Now, we shall prove that  $T$ is positive semidefinite. Let $x \in H$, then $x = x_1 + x_2$, where $x_1 \in N(T)$ and $x_2 \in N(T)^\perp =  \overline{R(T)}$. Now,  $\langle Tx, x \rangle =  \langle T(x_1+x_2), (x_1+x_2) \rangle  = \langle Tx_2, x_2 \rangle \geq 0$. Thus $T$ is positive semidefinite.
To complete the proof, let us show that $R(T)$ is closed. Assume that the sequence $\{Tx_n\}$  converges to $y$. We can assume that, the elements of the sequence $\{x_n\}$  are distinct and belong to $\overline{R(T)}$. Then, by Lemma \ref{mrT-closure}, we have $$m_r(T) \leq \frac{\langle Tx_n - T x_m, x_n-x_m \rangle}{||x_n - x_m||^2} \leq \frac{||Tx_n - T x_m||~||x_n-x_m||}{||x_n - x_m||^2} $$ so that  $$m_r(T) ||x_n - x_m|| \leq ||Tx_n - Tx_m||.$$ Since $\{Tx_n\}$ is a cauchy sequence, it follows that $\{x_n\}$ is also a cauchy  sequence and hence $\{x_n \}$ converges to some  $x \in H$. Thus $Tx_n$ converges to $Tx$, and hence $Tx=y$. Thus $R(T)$ is closed.
\end{proof}

Next, we show that the converse of Theorem \ref{posi-closed} is true. 

\begin{theorem}\label{posi-closed-conv}
Let $T \in \mathcal{B}(H)$ be a self-adjoint operator. If $T$ is positive semidefinite and $R(T)$ is closed, then  $m_r(T) > 0$.
\end{theorem}
\begin{proof}
Suppose that $m_r(T) =0$. Then, there exists a sequence $\{x_n\}$ in $R(T)$ such that $||x_n|| =1$ and $\langle Tx_n, x_n \rangle  \rightarrow 0$. We have $\langle Tx_n , x_n \rangle =  \langle T^{\frac{1}{2}}x_n , T^{\frac{1}{2}}x_n \rangle = ||T^{\frac{1}{2}}x_n||$ and so $T^{\frac{1}{2}}x_n \rightarrow 0$ which, in turn implies that $Tx_n \rightarrow 0$. Since $T$ has closed range, $T^\dagger$ is bounded and hence $T^\dagger T x_n \rightarrow 0$. But $T^\dagger T x_n = x_n$ for all $n$. So $x_n$ converges to $0$, which is not possible. Thus $m_r(T) >0$.
\end{proof}

Next, let us establish a property of closed range operators with $m_r(T)>0$.

\begin{theorem}\label{norm_rel}
Let $T \in \mathcal{B}(H)$ be a self adjoint operator with $m_r(T) >0 $. If
\begin{itemize}
\item[(a)] $R_1 = \sup \{\langle Tx, x \rangle : x \in R(T)$ and $||x|| =1 \},$
\item[(b)] $R_2 = \sup \{\langle Tx, x \rangle : x \in R(T)$ and $||x|| \leq 1 \},$
\item[(c)] $R_3 = \sup \{\langle Tx, x \rangle : ||x|| =1 \},$ and
\item[(d)] $R_4 = \sup \{\langle Tx, x \rangle : ||x|| \leq 1 \}, $
\end{itemize}
then $R_1 =R_2 =R_3 =R_4$.
\end{theorem}
\begin{proof}
It is easy to verify that $R_1 = R_2$ and $R_3 =R_4$.To complete the proof, let us prove $R_2 = R_4$. From the definition, it is clear that $R_4 \geq R_2$. Let $x \in H$ such that $||x||\leq 1.$ Then, $x = x_1 + x_2$ with $x_1 \in N(T)^\perp = R(T)$, $x_2 \in N(T)$ such that $||x_1|| \leq 1$ and $||x_2|| \leq 1$. Thus, for any  $x \in H$, we have $\langle Tx , x \rangle = \langle Tx_1,  x_1\rangle$ for some $x_1 \in R(T)$ with $||x_1|| \leq 1$. Thus, we have $R_4 = R_2$.
\end{proof}

In the next theorem, we derive relationships between $M(T^\dagger)$ and $m_r(T)$, and $M(T)$ and $m_r(T^\dagger)$

\begin{theorem}\label{mMinv}
Let $T \in \mathcal{B}(H)$ be a self adjoint operator such that $m_r(T)>0$. Then the following holds:
 \begin{enumerate}
 \item $M(T^\dagger) = [m_r(T)]^{-1}$, and
 \item $m_r(T^\dagger) = [M(T)]^{-1}.$
\end{enumerate}
\end{theorem}
\begin{proof}
Let $y \in R(T)$ with $||y||=1$. Then $y = Tx$, for some $x \in R(T)$. Now, $ \langle T^\dagger y, y \rangle = \langle T^\dagger T x, T x \rangle = \langle Tx, x \rangle$. Since, $m_r(T) \leq \frac{1}{||x||^2} \langle Tx, x\rangle$, we have $ \frac{1}{\langle Tx, x\rangle} \leq \frac{1}{||x||^2 m_r(T)}$ and hence $\langle Tx, x \rangle \leq \frac{\langle Tx , x \rangle^2}{||x||^2 m_r(T)} \leq \frac{||Tx||^2}{m_r(T)}$. Thus $\langle T^\dagger y, y\rangle \leq \frac{1}{m_r(T)}$, and hence $\sup\{\langle T^\dagger y , y \rangle: y \in R(T), ||y|| = 1\} \leq \frac{1}{m_r(T)}$. Now, by Theorem \ref{norm_rel}, we get $M(T^\dagger) \leq \frac{1}{m_r(T)}.$


If $x \in N(T)$ or $y \in N(T)$, then $\langle Tx, y \rangle = \langle x , Ty \rangle = 0.$ Let $x, y \in R(T)$, then, by Cauchy-Schwartz inequality for a semi inner product, we have $|\langle Tx, y \rangle|^2 \leq \langle Tx, x\rangle \langle Ty, y\rangle.$ If $y = T^\dagger x$, then $|\langle TT^\dagger x,  x \rangle|^2 \leq  \langle Tx, x\rangle \langle T^\dagger x, x\rangle$. Hence, we have $$\langle Tx, x \rangle \langle T^\dagger x, x \rangle \geq 1,$$ whenever $||x|| =1.$ Now, we have $\langle Tx, x \rangle  \geq \frac{1}{\langle T^\dagger x, x \rangle}$ and hence $\inf \{ \langle Tx, x \rangle: ||x||=1\}  \geq \inf \frac{1}{\{\langle T^\dagger x, x \rangle : ||x|| = 1\}} =  \frac{1}{\sup \{ \langle T^\dagger x, x \rangle: ||x|| = 1\}} $. Thus, we get $M(T^\dagger ) \geq \frac{1}{ m_r(T)}.$ The proof of the second assertion is similar.
\end{proof}

In the next corollary, we establish a relationship between $m_r(T)$ and norm of the operator $T^\dagger$.

\begin{cor}\label{comp-m-M}
If $T \in \mathcal{B}(H)$ is a self adjoint operator such that $m_r(T) >0 $, then $\displaystyle{||T^\dagger|| =\frac{1}{m_r(T)}}$.
\end{cor}
\begin{proof}
By the definition, $||T^\dagger|| =  \sup \{\displaystyle \langle T^\dagger x, x\rangle : x \in H, ||x|| =1 \}$.
 Thus, by Theorem  \ref{norm_rel}, we have $ ||T^\dagger||  = \displaystyle{\frac{1}{m_r(T)}}.$
\end{proof}

\section{Bounds for solution set of linear complementarity problems} \label{bound_lcp}
In this section, we establish bounds for a certain specific convex subset of the solution set of linear complementarity problems. For an operator $T \in \mathcal{B}(H)$ and  $b \in H$, the solution set of the associated linear complementarity problem, LCP$(T, b)$, is denoted by $SOL(T,b)$, is defined as the set of all solutions of the LCP$(T, b)$.

In general, the solution set of a linear complementarity problem need not be bounded. In the next theorem, we give a sufficient condition under which the solution is unbounded.

\begin{theorem}
Let $H$ be a real Hilbert space and let $K$ be a cone. Let $T \in \mathcal{B}(H)$  and $b \in H$. If  $ x \in N(T) \cap SOL(T, b)$ for some nonzero $x \in H$, then $SOL(T, b)$ is unbounded.
\end{theorem}
\begin{proof}
Let $ x \in N(T) \cap SOL(T, b)$ and $x  \neq 0$. Then $\alpha x \in N(T) \cap SOL(T, b)$, for all $ \alpha \geq 0.$ Hence $SOL(T, b)$ is unbounded.
\end{proof}

In the next theorem we establish a bound for those solutions which do not belong to the null space of the operator $T$.

\begin{theorem}\label{bounded1}
Let $H$ be a real Hilbert space and let $K$ be a cone. Let $T \in \mathcal{B}(H)$  and $b \in H$. If $m_r(T)>0$, then the solution set $SOL(T, b)\cap R(T^*)$ of $LCP(T, b)$ is a subset of $B(0, \frac{||b||}{m_r(T)}) \cap K$.
\end{theorem}
\begin{proof}For $x \in H$,
we have,
\begin{align*}
\langle Tx+b, x\rangle &=  \langle Tx, x\rangle + \langle b, x\rangle,      \\
 &\geq  m_r(T)||x||^2  - ||b|| || x||, \\
 & =  (m_r(T)||x|| - ||b||) || x||.
\end{align*}
Now, if $(m_r(T)||x|| - ||b||) >0 $, then $\langle Tx+b, x\rangle >0$. Thus $x \notin SOL(T, b)$.  Hence, if $ x \in SOL(T, b)\cap R(T^*) $, then $(m_r(T)||x|| - ||b||) \leq 0$. Thus $||x|| \leq  \frac{||b||}{m_r(T)}$,  which proves the claim.
\end{proof}

In the next theorem, we derive a bound for solution of the linear complementarity problem, whenever the solution vector belongs to the range space of the operator $T$.

\begin{theorem}
Let $H$ be a real Hilbert space and $K$ be a self-dual cone in $H$. If $T\in \mathcal{B}(H)$ is a self adjoint operator such that $m_r(T) >0 $, then for every solution $x \in R(T)$ of  $LCP(T, b)$ where $b \in R(T)$, one has $||x|| \leq \frac{M(T)}{m_r(T)} ||T^{\dagger}(b)||.$
\end{theorem}
\begin{proof}
Since $b\in R(T)$ and $T$ is self adjoint, we have $b=TT^\dagger b$ and $M(T)=\lvert \lvert T \rvert \rvert$. Thus $\lvert \vert b\rvert \rvert\leq \lvert \lvert T \rvert \rvert \lvert \lvert T^\dagger b \rvert \rvert$. Suppose that $x\in R(T)$ is a solution of $LCP (T,b)$. By Theorem \ref{bounded1},
\begin{align*}
\lvert \lvert x\rvert \rvert & \leq \frac{||b||}{m_r(T)},\\
&  \leq \frac{M(T)}{m_r(T)} ||T^{\dagger}(b)||.
\end{align*}
\end{proof}

The proof of the above theorem gives an alternate simple proof to \cite[Theorem 3.1]{isac-numerical}.

\begin{cor}
Let $H$ be a real Hilbert space and $K$ be a self-dual cone in $H$. If $T\in \mathcal{B}(H)$ is a self adjoint operator such that $m(T) >0 $, then for every solution $x$ of  $LCP(T, b)$, one has $||x|| \leq \frac{M(T)}{m(T)} ||T^{-1}(b)||.$
\end{cor}

\begin{theorem}\label{last}
Let $H$ be a real Hilbert space and let $T\in \mathcal{B}(H)$ be a self adjoint closed range operator. If $x$ is a  solution of $LCP(T, b)$, where $b \in R(T)$, then
\begin{center}
$\langle x-x_b, T(x-x_b)\rangle=\frac{1}{4}\langle b, T^\dagger b\rangle ,$
\end{center}
where $x_b=-\frac{1}{2} T^\dagger b$.
\end{theorem}
\begin{proof}
Let $b\in R(T)$ and $x_b=-\frac{1}{2}T^\dagger b$. Suppose that $x$ is a  solution of $LCP(T, b)$. Then
\begin{eqnarray}
\langle x-x_b, T(x-x_b)\rangle & =& \langle x,Tx\rangle + \frac{1}{2} \langle x,TT^\dagger b\rangle + \frac{1}{2} \langle T^\dagger b, T x\rangle + \frac{1}{4}\langle T^\dagger b, TT^\dagger b\rangle \nonumber \\
&=& \langle x,T x\rangle + \langle x, b\rangle +\frac{1}{4}\langle T^\dagger b, b\rangle \nonumber \\
&=&  \frac{1}{4}\langle b, T^\dagger b\rangle. \nonumber
\end{eqnarray}
\end{proof}

\begin{theorem}
Let $H$ be a real Hilbert space and  $T\in \mathcal{B}(H)$ be a self adjoint operator such that $m_r(T) >0 .$ Let $b \in R(T)$ and $x_b=-\frac{1}{2} T^\dagger b$.  Then for every solution $x\in R(T)$ of  $LCP(T, b)$, where $x\neq x_b$, one has
\begin{center}
$\frac{\| b \|}{2M(T)}\leq \| x-x_b \| \leq  \frac{\| b \|}{2m_r(T)}.$
\end{center}
\end{theorem}
\begin{proof}
By Theorem \ref{last}, we have $\langle x-x_b, T(x-x_b)\rangle=\frac{1}{4}\langle b, T^\dagger b\rangle.$
Also, $\frac{1}{4} m_r(T^\dagger)\|b\|^2\leq \|x-x_b\|^2M(T)$, by the definition of $M(T)$ and $m_r(T)$.
Now, by Theorem \ref{mMinv}, we have
$\frac{1}{4M(T)}\|b\|^2\leq \|x-x_b\|^2M(T).$
Thus, we can conclude that
\begin{equation}
\frac{1}{2M(T)}\|b\|\leq \|x-x_b\|. \label{1}
\end{equation}
Similarly, by Theorem \ref{mMinv} and \ref{last}, we have
$m_r(T)\|x-x_b\|^2\leq \frac{1}{4}M(T^{\dagger})\|b\|^2= \frac{1}{4m_r(T)}\|b\|^2$.
Thus \begin{equation}
\|x-x_b\|\leq \frac{1}{2m_r(T)}\|b\|. \label{2}
\end{equation}
From equation \eqref{1} and \eqref{2},
\begin{center}
$\frac{\| b \|}{2M(T)}\leq \| x-x_b \| \leq  \frac{\| b \|}{2m_r(T)}.$
\end{center}
\end{proof}



\subsection{Example}
Let $H = l^2(\mathbb{Z})$ and $H_+ = \{(\dots, x_{-2}, x_{-1},\fbox{$x_{0}$},x_{1},x_{2}, \dots): x_i \geq 0 ~\mbox{for all}~ i\}.$ Then $H_+$ is a self-dual cone.  Define $T : H \rightarrow H$ as $T((\dots, x_{-2}, x_{-1},\fbox{$x_{0}$},x_{1},x_{2}, \dots)) = (\dots, 0, 0, \fbox{$x_{0}$},x_{1},x_{2}, \dots)$. Then, $T$ is an orthogonal projection on to the space $l^2(\mathbb{N})$. So, $T = T^2 = T^*$ and $T = T^\dagger$. Consider  the vector $b = (\dots, 0, 0, \fbox{$0$}, \frac{-1}{3},1,\frac{1}{2}, \frac{1}{3}, \dots) \in R(T)$, it is easy to verify that $x = (\dots, 0, 0, \fbox{$0$}, \frac{1}{3},0,0, \dots)$ solves the associated LCP. It is clear that, $\Vert x \Vert \leq \Vert b \Vert$.  Also, $\Vert x - x_b \Vert = \frac{\sqrt{37}}{6} = \frac{\Vert b \Vert}{2}$. Indeed, in the vector $b$, if we may replace the entry $2$ by any real number.

It may be observed that, more generally, any orthogonal projection satisfies the assumptions of our theorem.\\

\textbf{Acknowledgement:}
Projesh Nath Choudhury  was partially supported by National Post-Doctoral Fellowship(PDF/2019/000275), the SERB, Department of Science and Technology, India,  and the NBHM Post-Doctoral Fellowship (0204/11/2018/R$\&$D-II/6437) from DAE (Govt. of India). 
M. Rajesh Kannan would like to thank the SRIC, IIT Kharagpur, the SERB, Department of Science and Technology, India, for financial support through the projects ISIRD,  MATRICS (MTR/2018/000986) and Early Career Research Award (ECR/2017/000643).

\bibliographystyle{amsplain}
\bibliography{gen-inv-lcp}

\end{document}